 \def\arxiv{}
\ifx\arxiv\undefined 
\documentclass[preprint]{elsarticle}
\usepackage{amssymb}
 \else
 \documentclass[final,1p,11pt]{elsarticle}
\usepackage{amssymb}
\makeatletter
\def\ps@pprintTitle{%
 \let\@oddhead\@empty
 \let\@evenhead\@empty
 \def\@oddfoot{\centerline{\thepage}}%
 \let\@evenfoot\@oddfoot}
\makeatother

\usepackage[T1]{fontenc}
\usepackage{charter}
\usepackage[expert]{mathdesign}
\usepackage{titlesec}
\titleformat{\section}[hang]{
    \usefont{T1}{bch}{b}{n}\selectfont} 
    {} 
    {0em} 
    {\hspace{-0.4pt}\Large \thesection\hspace{0.6em}} 
    [] 

\usepackage{tocloft}
\fi

\usepackage[active]{srcltx}
\usepackage[all]{xy} \xyoption{poly}

\usepackage[T1]{fontenc}
\usepackage[latin9]{inputenc}
\usepackage{array}
\usepackage{longtable}
\usepackage{multirow}
\usepackage{amsmath,amsthm}
\usepackage{enumitem}
\usepackage{graphicx}
\usepackage{float}
\usepackage{xspace}
\usepackage{amssymb}
\usepackage[english]{babel}
\usepackage{ifthen}
\usepackage{amscd}
\usepackage{psfrag}
\usepackage[hyperindex,bookmarksnumbered, plainpages,backref]{hyperref}

\usepackage{natbib}
\def\ka{\mathbf{k}}

\newcommand{\Rad}{\operatorname{Rad}}
\newcommand{\V}{\operatorname{V}}
\newcommand{\G}{\operatorname{G}}
\newcommand{\Aut}{\operatorname{Aut}}
\newcommand{\car}{\operatorname{char}}
\newcommand{\Ann}{\operatorname*{Ann}}
\newcommand{\Hom}{\operatorname{Hom}}

\newcommand{\J}{{\mathcal{J}}}
\newcommand{\Pe}{{\mathcal{P}}}
\newcommand{\B}{{\mathcal{B}}}
\newcommand{\T}{{\mathcal{J}}}

\newcommand{\Jor}{\operatorname{\J\hspace{-0.063cm} or}}

\newcommand{\dsum}{\displaystyle\sum}

\newtheorem{example}[equation]{Example}
\newtheorem{definition}[equation]{Definition}

\newtheorem{theorem}[equation]{Theorem}

\newtheorem{proposition}[equation]{Proposition}

\makeindex

\begin{document}

\title{The variety of three-dimensional real Jordan algebras}
\author[ime]{Iryna Kashuba\fnref{fn1}\corref{cor1}}
\ead{kashuba@ime.usp.br}

\author[ime]{Mar\'{\i}a Eugenia Martin\fnref{fn2}}
\ead{eugenia@ime.usp.br}

\cortext[cor1]{Corresponding author}
\fntext[fn1]{The author was supported by CNPq (309742/2013-7).}
\fntext[fn2]{The author was supported by the CAPES scholarship 
 for PostDoctoral program of Mathematics, IME-USP.}

\address[ime]{Instituto de Matem\'atica e Estat\'\i stica, Universidade de S\~ao Paulo, R. do Mat\~ao 1010, 05508-090, S\~ao Paulo, Brazil.}

\begin{abstract}
In this paper, we study the variety  $\Jor_{3}$ of three-dimensional Jordan algebras over
the field of real numbers. We establish the list of $26$ non-isomorphic
Jordan algebras and describe the irreducible components of $\Jor_{3}$ proving that it is the union of 
Zariski closure of the orbits of $8$ rigid algebras.
\end{abstract}

\begin{keyword}
 Jordan algebra \sep classification of algebras \sep deformation of algebras
\MSC[2010]{17C27, 17C55, 17C10.} 
\end{keyword}

\maketitle
\ifx\arxiv\undefined 
 \else
 \tableofcontents
 \newpage
 \fi
\section{Introduction}

Let $\V$ be an $n$-dimensional vector space over a field $\ka$ of characteristic $\neq2$, then the bilinear maps $\V\times\V\to\V$ form a vector space $\Hom_{\ka}(\V\times\V,\V)=\V^*\otimes\V^*\otimes\V$ of dimension $n^3$ which has structure of an affine variety $\ka^{n^3}$. The algebras satisfying the Jordan identity form a Zariski-closed affine subset of  $\ka^{n^3}$ which we call the variety of Jordan $\ka$-algebras of dimension $n$ and denote it by $\Jor_n$\footnote{Note that analogously one defines the varieties of associative and Lie algebras.}. Each point of $\Jor_n\subseteq\ka^{n^3}$ will be seen as an $n^3$-tuple of structure constants $(c_{ij}^k)$ with $i,j,k\in \{1,2,3\}$ and it represent a Jordan $\ka$-algebra of dimension $n$,  with respect to some fixed basis.

The linear general  group $\G=\operatorname{GL}(\V)$ operates on $\Jor_n$
by conjugation, decomposing the variety into $\G$-orbits
which correspond to the classes of isomorphic Jordan algebras. If an algebra $\J$ lies in the Zariski closure
of the orbit of a (non-isomorphic) algebra $\J'$ in the variety, then  we will
say that $\J'$ is a deformation of $\J$. An algebra $\J$ whose $\G$-orbit, $\J^{\G}$, is Zariski-open in $\Jor_n$
is called rigid. 

The goal of this paper is to classify algebraically and geometrically Jordan algebras of dimension three over the field $\mathbb{R}$ of real numbers. By geometric classification we mean the problem of determine the orbits of $\Jor_n$ under the action of $\G$, find all possible deformation between the algebras and, finally, determine the complete list of rigid algebras since the closure of the orbits of such algebras  generates an irreducible component of the variety.
 
In this work we study the variety $\Jor_3\subseteq\mathbb{R}^{27}$ of real Jordan algebras of dimension three. Firstly,  we work in the algebraic classification up to isomorphism of these algebras. This allowed us to determine that the number of $\G$-orbits in the variety $\Jor_3$ correspond to $26$. Next, given a three-dimensional real Jordan algebra we determined if it is rigid or we find a deformation of its to finally conclude that the variety $\Jor_3$ has $8$ irreducible components.

The classification problem of the algebraic structures of given dimension
has been extensively studied. The only class of algebras (associative,
Jordan and Lie) which is completely described in any dimension is the one of semi-simple
algebras. In general, the list of all algebras is known only for small
dimensions.

The literature, when the base field is algebraically closed is extensive. Namely, in 1975, P. Gabriel presented in \cite{finiterepresentatiotypeisopen}
the lists of all unitary associative algebras up to dimension four and described all rigid algebras of such variety. Later, G.
Mazzola, in his work \cite{mazzola}, extended the algebraic and geometric classification to dimension  five and he proved that there are ten irreducible components in this variety. In the context of Lie algebras both classification are known for dimension till six,  see \cite{kirillov}. As to Jordan algebras, H. Wesseler in \cite{wesseler} described unitary Jordan algebras up to dimension six and M. E. Martin in  \cite{martin} 
classified algebraically all Jordan algebras (associative and non-associative, unitary and non-unitary) up to dimension four.  Regarding the description of the variety $\Jor_{n}$ the references are rather recent. In \cite[2005]{irynashesta} I. Kashuba and I. Shestakov described the irreducible components of $\Jor_{3}$ and in \cite[2006]{iryna} the first author classified geometrically the unitary Jordan algebras of dimensions four and five.  In  \cite[2011]{ancocheabermudes}, the authors determined the laws and deformations of nilpotent Jordan algebras of dimension three and four over the field of complex numbers. Finally, in \cite[2014]{kashubamartin} I. Kashuba and M. E. Martin generalized these results to obtain a complete description for the varieties $\Jor_{n}$ for $n\leq4$.

When we consider the field of real numbers, the results in the literature for Jordan algebras (and even Lie
and associative algebras) are scarce, including for
small dimensions. In 2007, Ancochea Bermúdez and others, in their work \cite{AncocheaAssoc2D}, classified 
algebraically and geometrically associative algebras of dimension two, and afterwards, in \cite{dim2r},
the authors did an analogous study related to two-dimensional real Jordan algebras.

The paper is organized as follows. In Section \ref{sec:Preliminaries}, we recall  the basic concepts and necessary results for finite-dimensional
 Jordan algebras. Also, we present the methods used to study the irreducible components of the variety of Jordan algebras and the deformations between these algebras. In particular, we show that the dimension of the group of $2$-cocycles of an algebra does not increase under deformations. In Section \ref{sec:3-Jordan-algebras}, we classify all real Jordan algebras of dimension 
 three and also show that all algebras are pairwise non-isomorphic. Finally, in Section \ref{Geometric Classification} we construct
 deformations between algebras in $\Jor_3$ and describe its irreducible components.

\section{Preliminaries\label{sec:Preliminaries}}

In this section we start with the basic concepts, notations and principal
results about finite-dimensional Jordan algebras over a
field $\ka$ of characteristic $0$, and in the second part we 
will introduce the variety of $n$-dimensional Jordan 
algebras, $\Jor_n$, together with its properties.

For the standard terminology on Jordan algebras, the reader is referred to the book of N. Jacobson \cite{jacobson}, for concepts from deformation theory see \cite{gerstenhaber}.

\begin{definition} A \textbf{Jordan $\ka$-algebra} is a  commutative algebra
$\J$ with a multiplication \textquotedbl{}$\cdot$\textquotedbl{}
satisfying the Jordan identity: 
\begin{align}
((x\cdot x)\cdot y)\cdot x & =(x\cdot x)\cdot(y\cdot x),\qquad \text{for any } x,y\in\J\label{eq:identidadejor}
\end{align} 
or, equivalently, its linearization
\begin{equation}
(x,y,z\cdot w)+(w,y,z\cdot x)+(z,y,x\cdot w)=0\mbox{,}\label{eq:linearjord}
\end{equation}
for any $x,\mbox{ }y,\mbox{ }z,\mbox{ }w\in\J$. Here $(x,y,z):=(x\cdot y)\cdot z-x\cdot(y\cdot z)$
is the associator of $x,\, y,\, z$.
\end{definition}

\begin{example}Let $\left(\mathcal{U},j\right)$ be an associative
algebra with an involution $j$. Then
\[
H(\mathcal{U},j)=\left\{ u\in\mathcal{U}\mid u=j(u)\right\} ,
\]
the set of elements symmetric with respect to $j$ together with the multiplication given by the formula $x\odot y=\frac{1}{2}(x\cdot y+y\cdot x)$, 
where $\cdot$ is the multiplication in $\mathcal{U}$, is a Jordan algebra.
\end{example}

\begin{example}
Let $\,\V\,$ be a vector space over $\,\ka\,$ with a symmetric
bilinear form $\,f=f(x,y)$ on $\V$. Then the set
$\,\J(\V,f)=\ka 1+\V\,$ endowed with the multiplication
$$
(\alpha 1+x)(\beta 1+y)=(\alpha\beta+f(x,y))1+\alpha y+\beta
x\qquad \text{ for } \alpha,\beta\in\ka\ \ x,y\in \V
$$
forms a Jordan algebra called the Jordan algebra of the
symmetric bilinear form $\,f$.
\end{example}

For any Jordan algebra $\J$ we define inductively a series of subsets  by setting 
\begin{eqnarray*}
\J^{1}&=&\J^{\left\langle 1\right\rangle }=\J\text{,}\\
\J^{n}&=&\J^{n-1}\cdot\J+\J^{n-2}\cdot\J^{2}+\cdots+\J\cdot\J^{n-1},\\
\J^{\left\langle n\right\rangle }&=&\J^{\left\langle n-1\right\rangle }\cdot\J.
\end{eqnarray*}
For any $i\geq 1$,  both $\J^{i}$ and $\J^{\left\langle i\right\rangle }$ are ideals of the algebra $\J$. The chain 
$\J^{\left\langle 1\right\rangle }\supseteq\J^{\left\langle 2\right\rangle }\supseteq\cdots\supseteq\J^{\left\langle n\right\rangle }\supseteq\cdots$
is called the \textbf{lower
central series of $\J$}. The subset  $\J^{n}$ is called the \textbf{$n$-th power of the algebra
$\J$} . Observe that   $\J^{i}=\J^{\left\langle i\right\rangle }$, for $i=1,2,3$.

\begin{definition} A Jordan algebra $\J$ is called  \textbf{nilpotent} if there exists an integer $s$
such that $\J^{\left\langle s\right\rangle }=0$.
The minimal integer  for which this condition holds is the \textbf{nilindex}
of $\J$. \end{definition}

For a nilpotent algebra $\J$ of nilindex $s$ define the \textbf{nilpotency
type} of $\J$ as the sequence $(n_{1},n_{2},n_{3},\cdots,n_{s-1})$,
where $n_{i}=\dim\left(\J^{\left\langle i\right\rangle }/\J^{\left\langle i+1\right\rangle }\right)$.
Observe that all $n_{i}>0$. 


\begin{definition}
A Jordan algebra $\J$ is called:
\begin{enumerate}
 \item \textbf{simple} if $0$ and
$\J$ are the only ideals of $\J$ and $\J²\neq0$. 
\item\textbf{semi-simple}
if it is a direct sum of simple algebras.
\item\textbf{central simple} if $\J_K=\J\otimes_{\ka}K$ is simple for any extension $K$ of $\ka$.
\end{enumerate}
\end{definition} 

The following theorem classifies all finite-dimensional central simple
Jordan algebras.

\begin{theorem} \label{thm:dimn0}\cite[V.7]{jacobson} Let
$\J$ be a finite-dimensional central simple Jordan algebra over $\ka$. Then we have the following possibilities for
$\J$: 
\begin{enumerate}[label=\roman*.]
\item $\J=\ka$, 
\item $\J=\J(\V,f)$, the Jordan algebra of a non-degenerate symmetric bilinear
form $f$ on a finite-dimensional $\ka$-vector space $\V$ such that $\dim \V>1$, 
\item $\J=H(\mathcal{U},j)$, where $(\mathcal{U},j)$ is a finite-dimensional
central simple associative algebra with involution $j$ of degree $n\geq3$,
or
\item $\J$ is an algebra such that there exists a finite extension field
$K$ of the base field $\ka$ such that $\J_{K}\simeq H(M_{3}(\mathfrak{C}_{K}),\tau)$
where $M_{3}(\mathfrak{C}_{K})$ is the algebra of all $3\times3$
matrices with elements in a Cayley algebra $\mathfrak{C}$ over $K$
and $\tau$ is the standard involution conjugate transpose. 
\end{enumerate}
\end{theorem} 

The following proposition is known as the Wedderburn Principal Theorem,
in this case is formulated for finite-dimensional Jordan algebras over a field
of characteristic $0$.

\begin{proposition}\label{prop:ssmaisnilpo}\cite[p. 405]{penico}
Let $\J$ be a finite-dimensional Jordan algebra over a field $\ka$
of characteristic $0$, and let $N=\Rad(\J)$ be the radical of $\J$ (i.e. the unique maximal nilpotent ideal of $\J$).
Then there exists a subalgebra $\J_{ss}$ of $\J$ such that $\J=\J_{ss}\oplus N$
(as vector spaces) and $\J_{ss}\simeq\J/N$. \end{proposition}

Moreover, the quotient $\J_{ss}:=\J/N$ is semi-simple, has an identity element and its decomposition into simple components is unique, see \cite[V.2, V.5]{jacobson}. The identity element $\overline{e}$ of $\J_{ss}$ could be lifted to an idempotent $e$ of $\J$. Then, without loss of generality, we may assume  that all finite-dimensional  Jordan algebra over a field $\ka$ of $\car\ka=0$ either is nilpotent or has an idempotent element. Thus we have:

\noindent  \begin{theorem} \label{thm:Peirce1}\cite[III.1]{jacobson} Let $e$
be an idempotent  in
$\J$ then we have the following decomposition into a direct sum of subspaces
\[
\J=\Pe_{1}\oplus\Pe_{\frac{1}{2}}\oplus\Pe_{0},
\]
where $\Pe_{i}=\{x\in\J\mid x\cdot e=ix\}$, for
$i=0,\mbox{ }\frac{1}{2},\mbox{ }1$. 
 \end{theorem}

This decomposition is called the
\textbf{ Peirce decomposition} \textbf{of $\J$ relative to idempotent}
$e$. The multiplication table for the Peirce components $\Pe_i$ is: 
\begin{equation}
\begin{array}{c}
\Pe_{1}^{2}\subseteq\Pe_{1},\qquad\Pe_{1}\cdot\Pe_{0}=0,\qquad\Pe_{0}^{2}\subseteq\Pe_{0}\text{,}\\
\Pe_{0}\cdot\Pe_{\frac{1}{2}}\subseteq\Pe_{\frac{1}{2}},\qquad\Pe_{1}\cdot\Pe_{\frac{1}{2}}\subseteq\Pe_{\frac{1}{2}},\qquad\Pe_{\frac{1}{2}}^{2}\subseteq\Pe_{0}\oplus\Pe_{1}\text{.}
\end{array}\label{eq:pierce_one}
\end{equation}

 Furthermore, we have the following generalization: if
$\J$ is a Jordan algebra with an identity element which is a sum of pairwise orthogonal idempotents $e_{i}$, i.e.   $1=\sum_{i=1}^{n}e_{i}$,
we have the refined \textbf{Peirce decomposition} \textbf{of $\J$ relative
to idempotents }$\{e_{1},\ldots,e_{n}\}$: 
\begin{equation}
\J=\bigoplus_{1\leq i\leq j\leq n}\Pe_{ij}\label{eq:decopierce}
\end{equation}
where $\Pe_{ii}=\left\{ x\in\J\mid x\cdot e_{i}=x\right\} $
and $\Pe_{ij}=\left\{ x\in\J\mid x\cdot e_{i}=x\cdot e_{j}=\frac{1}{2}x\right\} $. 
The multiplication table for the Peirce components is:
\begin{equation}
\begin{array}{c}
\Pe_{ii}^{2}\subseteq\Pe_{ii},\qquad\Pe_{ij}\cdot\Pe_{ii}\subseteq\Pe_{ij},\qquad\Pe_{ij}^{2}\subseteq\Pe_{ii}\oplus\Pe_{jj}\text{,}\\
\Pe_{ij}\cdot\Pe_{jk}\subseteq\Pe_{ik},\qquad\Pe_{ii}\cdot\Pe_{jj}=\Pe_{ii}\cdot\Pe_{jk}=\Pe_{ij}\cdot\Pe_{kl}=0,
\end{array}\label{eq:pierce_many}
\end{equation}
where the indices $i$, $j$, $k$, $l$ are all different. Note that the Peirce decomposition is inherited for ideals of $\J$.

Let $\V$ be an $\,n$-dimensional $\,{\ka}$-vector space with a
fixed basis $\{e_1,e_2,\cdots,e_n\}$. To endow $\,\V\,$ with a Jordan $\,\ka$-algebra structure, $(\J,\cdot)$, it suffices to specify $\,n^3\,$ structure
constants $\,c_{ij}^k\in\ka$, namely,
$$
e_i\cdot e_j=\sum_{k=1}^n c_{ij}^k e_k, \qquad i,j
\in\{1,2,\dots,n\}.
$$
The choice of $\,c_{ij}^k\,$ is not arbitrary, it must reflect the
fact that Jordan algebras are commutative and
satisfy Jordan identity \eqref{eq:identidadejor}. Therefore we obtain:
\begin{equation}\label{J1}
\begin{array}{c}
c_{ij}^k=c_{ji}^k,\vspace{0,15cm}\\
\dsum_{a=1}^n c_{ij}^a\dsum_{b=1}^n c_{kl}^b c_{ab}^p-\dsum_{a=1}^n
c_{kl}^a\dsum_{b=1}^n c_{ja}^b c_{ib}^p+\dsum_{a=1}^n
c_{lj}^a\dsum_{b=1}^n c_{ki}^b
c_{ab}^p-\vspace{0,15cm}\\\dsum_{a=1}^n c_{ki}^a\dsum_{b=1}^n
c_{ja}^b c_{lb}^p+\dsum_{a=1}^n c_{kj}^a\dsum_{b=1}^n c_{il}^b
c_{ab}^p-\dsum_{a=1}^n c_{il}^a\dsum_{b=1}^n c_{ja}^b c_{kb}^p=0,
\end{array}
\end{equation}
for all $\,i,j,k,l,p\in\{1,2,\dots,n\}$. Polynomial equations \eqref{J1} cut
out an \textbf{algebraic variety} $\,\Jor_n\,$ in $\,\ka^{n^3}=
\V^*\otimes \V^*\otimes \V\,$. A point $\,(c_{ij}^k)\in \Jor_n\,$
represents an $\,n$-dimensional $\ka$-algebra $\,\J\,$ along
with a particular choice of basis (which gives the structure
constants $\,c_{ij}^k\,$). A change of basis in $\,\J\,$ gives
rise to a possible different point of $\,\Jor_n$ or, equivalently, the general linear group 
$\G=\operatorname{GL}(\V)$  operates on $\,\Jor_n\,$ via ``conjugation'':
\begin{equation}\label{gl-action}
g(\J,\cdot)\ \mapsto\ (\J,\cdot_g), \qquad x\cdot_g y= g(g^{-1}x\cdot
g^{-1}y),
\end{equation} for any
$\,\J\in \Jor_n\,$, $\,g\in {\G}\,$ and $\,x,y\in \V$. The
\textbf{$\,{\G}$-orbit} of a Jordan algebra $\,\J\,$, that is, the set of
all images of $\,\J\,$ under the action of $\,{\G}$, is denoted by
$\,\J^{\G}$. The set of different $\,{\G}$-orbits of this action is
in one-to-one correspondence with the set of isomorphism classes of
$\,n$-dimensional Jordan algebras. Now, we can consider the
inclusion diagrams of the Zariski closure of orbits of
$\,n$-dimensional Jordan algebras. To relate the orbits, we say
that $\,\J_1\,$ is  a \textbf{deformation} of $\,\J_2\,$ or that $\J_1$ \textbf{dominates} $\J_2$ and denote this by $\,\J_1\to
\J_2\,$, if the orbit $\, \J_2^{\G}\,$ is contained in the Zariski closure of
the orbit $\,\J_1^{\G}$. A Jordan algebra $\,\J\,$ is called {\it
rigid} if its $\G$-orbit, $\,\J^{\G}$, is a Zariski-open set in
$\,\Jor_n$. In terms of deformation if $\,\J_1\,$ is a deformation
of a rigid algebra $\,\J\,$ then $\,\J_1^{\G}\cap\J^{\G}\neq\varnothing$
and therefore $\,\J_1\simeq\J$.  

The rigid algebras are of
particular interest. Indeed, $\,\Jor_n\,$ as any affine variety
could be decomposed into its irreducible components. Then if
$\,\J\in \Jor_n\,$ is rigid, there exists an irreducible component
$T$  such that $\,\J^{\G}\cap T\,$ is a non-empty open subspace  in
$T$, and therefore the closure of $\J^{\G}$ contains $T$.

The most known sufficient condition for an algebra to be rigid is given in terms of its cohomology group. 
We say that the second cohomology group $H^2(\J,\J)$ of a Jordan algebra $\J$ with coefficients in itself 
vanishes if for every bilinear mapping $h:\J\times\J\to\J$ satisfying
\begin{equation}\label{factorset}
\begin{array}{c}
h(a,b)=h(b,a)\\
(h(a,a)b)a+h(a^{2},b)a+h(a^{2}b,a)=a^{2}h(b,a)+h(a,a)(ba)+h(a^{2},ba)
 \end{array}
\end{equation}
for all $a,b\in\J$ there exists a linear mapping $\mu:\J\to\J$ such that
\begin{equation}\label{hmu}
 h(a,b)=\mu(ab)-a\mu(b)-\mu(a)b\,.
\end{equation}
The bilinear mapping $h$ is called \textbf{$2$-cocycle} of $\J$ and we denote the set of all $2$-cocycles of $\J$ by $Z^2(\J,\J)$. For the precise definition of these groups for Jordan algebras we refer to \cite{jacobson}. 

\begin{proposition}\label{cohomology} If the second cohomology group $H^2(\J,\J)$ of a Jordan algebra $\J$ with coefficients in itself 
vanishes, then $\J$ is rigid. In particular it follows that any semi-simple 
Jordan algebra is rigid. 
\end{proposition}

It was originally obtained 
in \cite{gerstenhaber} for associative and Lie algebras, for the case of Jordan algebras the proof 
is analogous, see \cite{iryna}.

\begin{example}\label{T12rigida}
Consider $\T_{12}\in\Jor_3$ the Jordan algebra with basis $\{e_1,n_1,n_2\}$ and multiplication given by $e_{1}^{2}=e_{1}$ and  $e_{1}\cdot n_{i}=\frac{1}{2}n_{i}\;$ for $i=1,2$. Let $h:\T_{12}\times\T_{12}\to\T_{12}$
be a bilinear map satisfying $\eqref{factorset}$, then 
\[
h(e_{1},e_{1})=\alpha e_{1},\quad h(e_{1},n_{i})=\beta_i e_{1}+\frac{\alpha}{2}n_{i},\quad h(n_{i},n_{j})=\beta_j n_{i}+\beta_i n_{j}
\]
for any $\alpha,\beta_i\in\mathbb{R}$ and $1\leq i,j\leq2$. Define a linear mapping $\mu:\T_{12}\to\T_{12}$
as been $\mu(e_{1})=-\alpha e_{1}$ and $\mu(n_{i})=-2\beta_i e_{1}+n_{i}$ for $i=1,\,2$, then
$\eqref{hmu}$ holds and thus $H^{2}(\T_{12},\T_{12})=0$,
which implies $\T_{12}$ is rigid.
 \end{example}

We will construct the deformations between Jordan algebras using the
following property. Let $\{e_1,\cdots,e_n\}$ be a basis of $\,\J\,$
as a vector space. Then the multiplication in algebra $\,\J\,$ is
defined by $\,n^3\,$ structure constants $\,c_{ij}^k$. Let
\begin{equation}\label{matrix}
g(t)\in {\rm Mat_n(\ka[t])} \end{equation} be a change of basis of $\,\J\,$ such that for any $\,t\neq 0\,$ it is
non-degenerate, i.e. $\,g(t)\in \G$. We denote by $\,\J_t\,$ the
algebra obtained from $\,\J\,$ by the change of basis $\,g(t)$
and let $\,c_{ij}^k(t)\,$ denote the corresponding structure
constants. Note that by \eqref{gl-action} $\,\J_t\simeq \J\,$ for
any $\,t\ne 0$. Then if $\,\J_1\,$ is a Jordan algebra defined by
structure constants $d_{ij}^k=c_{ij}^k(0)$ with respect to the same basis
$\{e_1,\cdots,e_n\}$ then $\,\J_1\in \overline{\J^{\G}}\,$ and thus  $\,\J\,$ is a
deformation of $\,\J_1$.

In the following proposition we collect properties which will be used to 
show that there is no deformation between certain algebras.

\begin{proposition}\label{conditions}
 Let $\J, \J_1\in \Jor_n$ and $\J\to \J_1$. Then 
\begin{enumerate}[label=(\roman*)]
\item \label{Aut} $\,\dim\Aut(\J)<
\dim\Aut(\J_1)$, where $\Aut(\J)< \G$ is 
the automorphism group of $\J$. 
\item \label{Rad} $\dim\Rad(\J)\leq\dim\Rad(\J_1)$.
\item  \label{Ann} $\dim\Ann(\J)\leq\dim\Ann(\J_1)$, where $\operatorname*{Ann}(\J)=\left\{ a\in\J\mid a\J=0\right\} $ is the 
annihilator of $\J$.

\item \label{potencia} $\dim\J^r\geq\dim\J_1^r$, for any positive integer $r$.
\item\label{direct_sum}If also $\J',\J'_1\in\Jor_{n'}$ and $\J'\to\J'_1$, then $\J\oplus\J'\to\J_1\oplus\J'_1$.
\item \label{identidade}Any polynomial identity of $\,\J$ is valid
in
 $\,\J_1$.  In particular, any deformation of a non-associative
 algebra  is again non-associative.
 \item\label{dimZ2}$\dim Z^2(\J,\J)\leq\dim Z^2(\J_1,\J_1)$.

\end{enumerate}
\end{proposition}

\begin{proof}
For the proof of \ref{Aut} to \ref{identidade} we refer to
\cite{kashubamartin}. To see \ref{dimZ2}, let $\{e_1,e_2,\cdots,e_n\}$ be a basis for $\J$ and $h\in Z^2(\J,\J)$. Then $h$ is completely define by the $n^{3}$ constants $\alpha_{ij}^{k}\in\mathbb{R}$
given by $h(e_{i},e_{j})=\sum_{k=1}^{n}\alpha_{ij}^{k}e_{k}$. 

Linearizing 
the identity of $2$-cocycle \eqref{factorset}
we have:
\begin{eqnarray*}
&&(h(x,y)w)z+(h(x,z)w)y+(h(y,z)w)x+h((xy)w,z)+h((xz)w,y)+\\
&&+h((yz)w,x)+h(xy,w)z+h(xz,w)y+h(yz,w)x=(xy)h(w,z)+\\
&&+(xz)h(w,y)+(yz)h(w,x)+h(x,y)(wz)+h(x,z)(wy)+h(y,z)(wx)+\\
&&+h(xy,wz)+h(xz,wy)+h(yz,wx).\\
\end{eqnarray*}
Computing this identity and the commutative condition of $h$ in the basis, it results in $l(n)=\left(n^{4}+\frac{n(n-1)}{2}\right)n$
equations having  $\alpha_{ij}^{k}$ as unknowns. Then $\dim Z^{2}(\J,\J)=n^{3}-\operatorname{rank}(P_{l(n),n^{3}})$,
where $P_{l(n),n^{3}}$ represent the matrix of the system of equations, thus $\dim Z^{2}(\J,\J)\geq s$ is equivalent to the fact that all $(n^3-s+1)$-minors of $P_{l(n),n^{3}}$ vanish. Then the set $\left\{ \J\in\Jor_{n}\mid\dim Z^{2}(\J,\J)\geq s\right\} $
is Zariski-closed.

\end{proof}

\section{Real Jordan algebras of small dimensions \label{sec:3-Jordan-algebras}}

In this section we present the lists of all one and two dimensional
indecomposable Jordan algebras over $\mathbb{R}$. Further we describe all  non-isomorphic,
three-dimensional real Jordan algebras. 

Recall that  $\J=\J_{ss}\oplus N$, where $N$ denotes the radical
of $\J$ and $\J_{ss}$ is semi-simple. We will denote by $e_{i}$
the elements in $\J_{ss}$ and by $n_{i}$ the ones which belong to
$N$. Henceforth, for
convenience we drop $\cdot$ and denote the multiplication in $\J$ simply 
as $xy$.

\subsection{Real Jordan algebras of dimension one\label{sec:apenddim1}}

There are two non-isomorphic one-dimensional real Jordan algebras: the
simple algebra $\mathbb{R}e$, with $e^{2}=e$ and the nilpotent
algebra $\mathbb{R}n$, where $n^{2}=0$.

\subsection{Real Jordan algebras of dimension two\label{sec:apenddim2}}

In \cite{dim2r} all two dimensional real Jordan algebras are described. Using their list 
we have the following $4$ indecomposable algebras:

\begin{longtable}{|c|c|c|}
\hline 
$\B$ & Multiplication Table & Observation\tabularnewline
\hline
\endfirsthead
\hline
$\B$ & Multiplication Table & Observation\tabularnewline
\endhead
\hline 
$\B_{1}$ & $e_{1}^{2}=e_{1}\quad e_{1}n_{1}=n_{1}\quad n_{1}^{2}=0$ & associative\tabularnewline
\hline 
$\B_{2}$ & $e_{1}^{2}=e_{1}\quad e_{1}n_{1}=\frac{1}{2}n_{1}\quad n_{1}^{2}=0$ & \tabularnewline
\hline 
$\B_{3}$ & $n_{1}^{2}=n_{2}\quad n_{1}n_{2}=0\quad n_{2}^{2}=0$ & associative, nilpotent\tabularnewline
\hline 
$\B_{4}$ & $e_{1}^{2}=e_{1}\quad e_{1}e_{2}=e_{2}\quad e_{2}^{2}=-e_{1}$ & associative, simple\tabularnewline
\hline

\caption{Indecomposable two-dimensional Jordan algebras over $\mathbb{R}$.}
\end{longtable}

\subsection{Real Jordan algebras of dimension three\label{sec:apenddim3}}
 
The description of three-dimensional real Jordan algebras is organized according to
the dimension of the radical and subsequently the possible values
of the nilpotency type. Also for each algebra we calculate the dimensions
of its automorphism group $\Aut(\T)$ and the annihilator
$\Ann(\T)$. 

We denote by $\J^{\#}=\J\oplus\mathbb{R}1$ the Jordan algebra obtained
by formal adjoining of the identity element $1$ of $\mathbb{R}$.

\subsubsection{Semisimple Jordan algebras}

Any simple Jordan algebra can be considered as a central simple algebra
over its centroid which is a field, see \cite[p.13]{Schafer}. As a consequence we can reduce the problem of classify finite dimensional simple Jordan algebras over $\mathbb{R}$ to the problem of classify central simple ones over a finite extension of $\mathbb{R}$, i. e. $\mathbb{C}$. Thus, by Theorem
\ref{thm:dimn0} if $\T$ is a simple Jordan algebra of dimension $\leq3$ over $\mathbb{R}$ then:
\begin{enumerate}[label=\roman*)]
 \item $\T=\mathbb{R}e$ of dimension one,
 \item $\T=\B_4$ of dimension two (that is the field of the complex numbers),
 \end{enumerate}
And if $\V$ is a two-dimensional real vector space with basis $\{e_1,e_2\}$ then:
\begin{enumerate}[resume,label=\roman*)]
 \item $\T=\J(\V,f_1)$ of dimension three, where $f_1$ is the non-degenerate symmetric bilinear form: $f_1(e_1,e_1)=1$, $f_1(e_2,e_2)=-1$ and $f_1(e_1,e_2)=0$,
 \item $\T=\J(\V,f_2)$ of dimension three, where $f_2$ is given by: $f_2(e_i,e_i)=-1$ for $i=1,2$ and $f_2(e_1,e_2)=0$,
 \item $\T=\J(\V,f_3)$ of dimension three, where $f_3$ is given by: $f_3(e_i,e_i)=1$ for $i=1,2$ and $f_3(e_1,e_2)=0$.
\end{enumerate}
Considering direct sum of these algebras we obtain all semi-simple real Jordan algebras of dimension three:

\begin{longtable}{|>{\centering}m{0.32cm}|>{\centering}m{5.3cm}|>{\centering}m{1.03cm}|>{\centering}m{1.04cm}|>{\centering}m{1.8cm}|}

\hline 
$\J$ & Multiplication Table & $\dim$ $\operatorname*{Aut}(\J)$ & $\dim$ $\operatorname*{Ann}(\J)$ & Observation\tabularnewline
\hline 
\endfirsthead
\hline
$\J$ & Multiplication Table & $\dim$ $\operatorname*{Aut}(\J)$ & $\dim$ $\operatorname*{Ann}(\J)$ & Observation\tabularnewline
\endhead
\hline 
$\T_{1}$ & $\mathbb{R}e_{1}\oplus\mathbb{R}e_{2}\oplus\mathbb{R}e_{3}$ & $0$ & $0$  & associative unitary\tabularnewline
\hline 
$\T_{2}$ & $\B_{4}\oplus\mathbb{R}e_{3}$ & $0$ & $0$  & associative unitary\tabularnewline
\hline 
$\T_{3}$ & $e_{2}^{2}=e_{1}\quad e_{3}^{2}=-e_{1}$ 

$e_{1}e_{i}=e_{i}\;$\scriptsize{$i=1,2,3$} & $1$ & $0$  & unitary $\J(\V,f_1)$\tabularnewline
\hline 
$\T_{4}$ & $e_{2}^{2}=e_{3}^{2}=-e_{1}\quad e_{1}e_{i}=e_{i}\;$\scriptsize{$i=1,2,3$} & $1$ & $0$  & unitary $\J(\V,f_2)$ \tabularnewline
\hline 
$\T_{5}$ & $e_{2}^{2}=e_{3}^{2}=e_{1}\quad e_{1}e_{i}=e_{i}\;$\scriptsize{$i=1,2,3$} & $1$ & $0$  & unitary $\J(\V,f_3)$\tabularnewline
\hline 
\caption{Three-dimensional semi-simple Jordan algebras over $\mathbb{R}$.}
\end{longtable}

\subsubsection{Jordan algebras with one-dimensional radical}

Consider $\T$ a real Jordan algebra of dimension three with $\dim N=1$. Thus $\J_{ss}$ is two-dimensional and by
Sections \ref{sec:apenddim1} and \ref{sec:apenddim2} we have the following possibilities:

\vspace{0.1cm}

\noindent \textbf{1). $\J_{ss}=\mathbb{R}e_{1}\oplus\mathbb{R}e_{2}$}.
Then $\J^{\#}=\J\oplus\mathbb{R}1$ contains $3$
orthogonal idempotents $e_{1}$, $e_{2}$ and $e_{0}=1-e_{1}-e_{2}$,
so using the Peirce decomposition \eqref{eq:decopierce} we have: 
\[
\J=\Pe_{00}\oplus\Pe_{01}\oplus\Pe_{02}\oplus\Pe_{11}\oplus\Pe_{12}\oplus\Pe_{22},
\]
and the corresponding decomposition of the ideal $N$: 
\[
N=N_{00}\oplus N_{01}\oplus N_{02}\oplus N_{11}\oplus N_{12}\oplus N_{22},
\]
where $N_{ij}=N\cap\Pe_{ij}$. Let $n_{1}$ be a basis of $N$, then 
$\J$ is completely defined by the subspace $N_{ij}$ to which belongs $n_1$.
Thus $\T$ is one of the following algebras:

\begin{longtable}{|>{\centering}m{0.32cm}|>{\centering}m{5.3cm}|>{\centering}m{1.03cm}|>{\centering}m{1.04cm}|>{\centering}m{1.8cm}|}
\hline 
$\J$ & Multiplication Table & $\dim$ $\operatorname*{Aut}(\J)$ & $\dim$ $\operatorname*{Ann}(\J)$ & Observation\tabularnewline
\hline 
\endfirsthead
\hline 
$\J$ & Multiplication Table & $\dim$ $\operatorname*{Aut}(\J)$ & $\dim$ $\operatorname*{Ann}(\J)$ & Observation\tabularnewline
\endhead
\hline 
$\T_{6}$ & $\mathbb{R}e_{1}\oplus\mathbb{R}e_{2}\oplus\mathbb{R}n_{1}$ & $1$ & $1$ & associative $n_{1}\in N_{00}$\tabularnewline
\hline 
$\T_{7}$ & $\B_{2}\oplus\mathbb{R}e_{2}$ & $2$ & $0$ & $n_{1}\in N_{01}$\tabularnewline
\hline 
$\T_{8}$ & $e_{i}^{2}=e_{i}\quad e_{i}n_{1}=\frac{1}{2}n_{1}\,$\scriptsize{$i=1,2$} & $2$& $0$ & unitary $n_{1}\in N_{12}$\tabularnewline
\hline 
$\T_{9}$ & $\B_{1}\oplus\mathbb{R}e_{2}$ & $1$ & $0$ & associative unitary $n_{1}\in N_{11}$\tabularnewline
\hline

\caption{Three-dimensional Jordan algebras over $\mathbb{R}$ with one-dimensional radical and
semi-simple part $\J_{ss}=\mathbb{R}e_{1}\oplus\mathbb{R}e_{2}$.}
\end{longtable}

\noindent \textbf{2). $\J_{ss}=\B_{4}$}. It contains
only one idempotent $e_{1}$ which determines the following decomposition
of $N$: 
\[
N=N_{0}\oplus N_{1}\oplus N_{\frac{1}{2}}.
\]
Let $n_{1}$ be a basis of $N$. If $n_{1}\in N_{0}$, then we obtain that
$e_{2}n_{1}\in\Pe_{1}\Pe_{0}=0$. If $n_{1}\in N_{1}$,
we have $e_{2}n_{1}=\alpha n_{1}$. Substituting $\left\{ e_{1},\, e_{2},\, n_{1}\right\} $ into
the Jordan identity \eqref{eq:linearjord} we obtain
that $\alpha=0$. Finally there is no real Jordan algebra with $n_{1}\in N_{\frac{1}{2}}$.

Thus we conclude that $\T$ is one of the following algebras:

\begin{longtable}{|>{\centering}m{0.32cm}|>{\centering}m{5.3cm}|>{\centering}m{1.03cm}|>{\centering}m{1.04cm}|>{\centering}m{1.8cm}|}

\hline 
$\J$ & Multiplication Table & $\dim$ $\operatorname*{Aut}(\J)$ & $\dim$ $\operatorname*{Ann}(\J)$ & Observation\tabularnewline
\hline 
\endfirsthead
\hline 
$\J$ & Multiplication Table & $\dim$ $\operatorname*{Aut}(\J)$ & $\dim$ $\operatorname*{Ann}(\J)$ & Observation\tabularnewline
\endhead
\hline 
$\T_{10}$ & $\B_{4}\oplus\mathbb{R}n_{1}$ & $1$ & $1$ & associative $n_{1}\in N_{0}$ \tabularnewline
\hline 
$\T_{11}$ & $e_{2}^{2}=-e_{1}\quad e_{1}n_{1}=n_{1}$

$e_{1}e_{i}=e_{i}\;$\scriptsize{$i=1,2$} & $2$ & $0$ & unitary $n_{1}\in N_{1}$ \tabularnewline
\hline 

\caption{Three-dimensional Jordan algebras over $\mathbb{R}$ with one-dimensional radical and
semi-simple part $\J_{ss}=\B_{4}$.}

\end{longtable}

\subsubsection{Jordan algebras with two-dimensional radical}
Now, suppose that $\T$ is a three-dimensional real Jordan algebra with $\dim N=2$. The only semi-simple one-dimensional Jordan algebra is $\J_{ss}=\mathbb{R}e_{1}$,
therefore we have the following Peirce decomposition of $N$: 
\[
N=N_{0}\oplus N_{\frac{1}{2}}\oplus N_{1}\text{.}
\]
The ideal $N$ may have two nilpotency types: $(2)$ or $(1,1).$

\vspace{0.1cm}

\noindent \textbf{1). Nilpotency type $(2)$.} Then $N^{2}=0$. Let
$\{n_{1},n_{2}\}$ be a basis of $N$, then it is enough to choose
to which Pierce components belong $n_{1}$ and $n_{2}$. Thus $\T$ is one of the following algebras: 

\begin{longtable}{|>{\centering}m{0.32cm}|>{\centering}m{4.1cm}|>{\centering}m{1.03cm}|>{\centering}m{1.04cm}|>{\centering}m{3cm}|}

\hline 
$\J$ & Multiplication Table & $\dim$ $\operatorname*{Aut}(\J)$ & $\dim$ $\operatorname*{Ann}(\J)$ & Observation\tabularnewline
\hline 
\endfirsthead
\hline 
$\J$ & Multiplication Table & $\dim$ $\operatorname*{Aut}(\J)$ & $\dim$ $\operatorname*{Ann}(\J)$ & Observation\tabularnewline
\endhead
\hline 
$\T_{12}$ & $e_{1}^{2}=e_{1}\quad e_{1}n_{i}=\frac{1}{2}n_{i}\;$

\scriptsize{$i=1,2$} & $6$ & $0$ & $n_{1},\, n_{2}\in N_{\frac{1}{2}}$\tabularnewline
\hline 
$\T_{13}$ & $e_{1}^{2}=e_{1}\quad e_{1}n_{i}=n_{i}\;$\scriptsize{$i=1,2$} & $4$ & $0$ & associative, unitary $n_{1},\, n_{2}\in N_{1}$\tabularnewline
\hline 
$\T_{14}$ & $\B_{2}\oplus\mathbb{R}n_{2}$ & $3$ & $1$ & $n_{2}\in N_{0},\, n_{1}\in N_{\frac{1}{2}}$\tabularnewline
\hline 
$\T_{15}$ & $\B_{1}\oplus\mathbb{R}n_{2}$ & $2$ & $1$ & associative $n_{2}\in N_{0},\, n_{1}\in N_{1}$\tabularnewline
\hline 
$\T_{16}$ & $e_{1}^{2}=e_{1}\quad e_{1}n_{1}=\frac{1}{2}n_{1}$

$\quad e_{1}n_{2}=n_{2}$ & $3$ & $0$ & $n_{1}\in N_{\frac{1}{2}},\, n_{2}\in N_{1}$\tabularnewline
\hline 
$\T_{17}$ & $\mathbb{R}e_{1}\oplus\mathbb{R}n_{1}\oplus\mathbb{R}n_{2}$ & $4$ & $2$ & associative $n_{1},\, n_{2}\in N_{0}$ \tabularnewline
\hline 

\caption{Three-dimensional Jordan algebras over $\mathbb{R}$ with two-dimensional radical of nilpotency
type $(2)$. }
\end{longtable}

\noindent \textbf{2).} \textbf{Nilpotency type $(1,1)$.} There exists
$n\in N$ such that $N=\mathbb{R}n+\mathbb{R}n^{2}$ with $n^{3}=0$.
Suppose, firstly, that $N=N_{i}$ then $n^{2}\in N_{i}^{2}\subseteq N_{0}\oplus N_{1}$,
that implies $i=0$ or $i=1$. Now, suppose that $N=N_{i}\oplus N_{\frac{1}{2}}$
with $i=0,1$ and $\dim N_{i}=\dim N_{\frac{1}{2}}=1$, we can choose
$n$ as an element of $N_{\frac{1}{2}}$. In fact if $N_{i}=\mathbb{R}a$
and $N_{\frac{1}{2}}=\mathbb{R}b$, then we have $b^{2}\in N_{i}$
thus $b^{2}=\alpha a$ for some $\alpha\in\mathbb{R}$. Note that
$\alpha\neq0$ since by nilpotency $a^{2}=ab=0$. Consequently $N=\mathbb{R}b\oplus\mathbb{R}b^{2}$.

Finally, if $N=N_{0}\oplus N_{1}$, then by nilpotency $N_{0}^{2}=N_{1}^{2}=0$.
Moreover $N_{0}N_{1}=0$ and thus $N^{2}=0$ leads to contradiction. 

Therefore we obtain that $\T$ is one of the following algebras,
where $n_{1}^{2}=n_{2}$.

\begin{longtable}{|>{\centering}m{0.32cm}|>{\centering}m{4.1cm}|>{\centering}m{1.03cm}|>{\centering}m{1.04cm}|>{\centering}m{3cm}|}

\hline 
$\J$ & Multiplication Table & $\dim$ $\operatorname*{Aut}(\J)$ & $\dim$ $\operatorname*{Ann}(\J)$ & Observation\tabularnewline
\hline 
\endfirsthead
\hline 
$\J$ & Multiplication Table & $\dim$ $\operatorname*{Aut}(\J)$ & $\dim$ $\operatorname*{Ann}(\J)$ & Observation\tabularnewline
\endhead
\hline 
$\T_{18}$ & $e_{1}^{2}=e_{1}\quad n_{1}^{2}=n_{2}$

$ e_{1}n_{i}=n_{i}\;$\scriptsize{$i=1,2$} & $2$ & $0$ & associative, unitary $n_{1},\, n_{2}\in N_{1}$\tabularnewline
\hline 
$\T_{19}$ & $e_{1}^{2}=e_{1}\quad n_{1}^{2}=n_{2}$

$e_{1}n_{1}=\frac{1}{2}n_{1}$ & $2$ & $1$ & $n_{1}\in N_{\frac{1}{2}},$

$n_{2}\in N_{0}$\tabularnewline
\hline 
$\T_{20}$ & $e_{1}^{2}=e_{1}\quad n_{1}^{2}=e_{1}n_{2}=n_{2}$

$ e_{1}n_{1}=\frac{1}{2}n_{1}$ & $2$ & $0$ & $n_{1}\in N_{\frac{1}{2}},$

$n_{2}\in N_{1}$\tabularnewline
\hline 
$\T_{21}$ & $\B_{3}\oplus\mathbb{R}e_{1}$ & $2$ & $1$ & associative $n_{1},\, n_{2}\in N_{0}$\tabularnewline
\hline

\caption{Three-dimensional Jordan algebras over $\mathbb{R}$ with two-dimensional radical of nilpotency
type $(1,1)$.}

\end{longtable}

\subsubsection{Nilpotent Jordan algebras}

We have the following nilpotency types for $\J$:

\noindent \textbf{1).} \textbf{Nilpotency type $(3)$.} Then $\J^{2}=0$.

\begin{longtable}{|>{\centering}m{0.32cm}|>{\centering}m{5.3cm}|>{\centering}m{1.03cm}|>{\centering}m{1.04cm}|>{\centering}m{1.8cm}|}

\hline 
$\J$ & Multiplication Table & $\dim$ $\operatorname*{Aut}(\J)$ & $\dim$ $\operatorname*{Ann}(\J)$ & Observation\tabularnewline
\hline 
\endfirsthead
\hline 
$\J$ & Multiplication Table & $\dim$ $\operatorname*{Aut}(\J)$ & $\dim$ $\operatorname*{Ann}(\J)$ & Observation\tabularnewline
\endhead
\hline 
$\T_{22}$ & $\mathbb{R}n_{1}\oplus\mathbb{R}n_{2}\oplus\mathbb{R}n_{3}$ & $9$ & $3$ & associative \tabularnewline
\hline

\caption{Three-dimensional nilpotent real Jordan algebra of nilpotency type $(3)$.}
\end{longtable}
\textbf{2). Nilpotency type $(1,1,1)$.} Then $\J^{\left\langle 4\right\rangle }=0$,
$\dim\J^{3}=1$ and $\dim\J^{2}=2$. We claim that there exists $n\in\J$
such that $\J=\mathbb{R}n+\mathbb{R}n^{2}+\mathbb{R}n^{3}$. Indeed, 
let $n_{1}$ be a basis for $\J^{3}$,  then complete it to a basis of $\J^2$ and $\J$, there is  $n_{2}\in\J^{2}$
and $n_{3}\in\J$, such that $\J^{2}=\mathbb{R}n_{1}+\mathbb{R}n_{2}$
and $\J=\mathbb{R}n_{1}+\mathbb{R}n_{2}+\mathbb{R}n_{3}$ where
\begin{gather*}
n_{1}^{2},\, n_{1}n_{2},\, n_{1}n_{3}\in\J^{\left\langle 4\right\rangle }=0,\quad n_{3}^{2}=\alpha n_{1}+\beta n_{2},\mbox{\ since }n_{3}^{2}\in\J^{2},\\
n_{2}n_{3}=\gamma n_{1}\mbox{ and }n_{2}^{2}=\delta n_{1},\mbox{\  since }n_{2}n_{3},\, n_{2}^{2}\in\J^{3}.
\end{gather*}

Substituting  $\left\{ n_{1},\, n_{2},\, n_{3}\right\} $  into
the Jordan identity \eqref{eq:linearjord} we obtain
either $\beta=0$ or $\delta=0$. But if $\beta=0$ then $\dim\J^{2}=1$, 
thus we have $\beta\neq0$ and $\delta=0$, analogously $\gamma\neq0$.
Then $n=n_3$ and $\{n,n^2,n^3\}$ is the desired basis. Thus we obtain the following algebra:

\begin{longtable}{|>{\centering}m{0.32cm}|>{\centering}m{5.3cm}|>{\centering}m{1.03cm}|>{\centering}m{1.04cm}|>{\centering}m{1.8cm}|}

\hline 
$\J$ & Multiplication Table & $\dim$ $\operatorname*{Aut}(\J)$ & $\dim$ $\operatorname*{Ann}(\J)$ & Observation\tabularnewline
\hline 
\endfirsthead
\hline 
$\J$ & Multiplication Table & $\dim$ $\operatorname*{Aut}(\J)$ & $\dim$ $\operatorname*{Ann}(\J)$ & Observation\tabularnewline
\endhead
\hline 
$\T_{23}$ & $n_{2}n_{3}=n_{1}\quad n_{3}^{2}=n_{2}$ & $3$ & $1$ & associative\tabularnewline
\hline 

\caption{Three-dimensional nilpotent  real Jordan algebra of nilpotency type $(1,1,1)$.}
\end{longtable}

\textbf{3).} \textbf{Nilpotency type $(2,1)$.} We observe that
$\J^{3}=0$ and $\dim\J^{2}=1$. Let $n_{3}$ be a basis of $\J^{2}$, complete it to a basis of $\J$
choosing some $n_{1},\: n_{2}\in\J$, we have $\J=\mathbb{R}n_{1}+\mathbb{R}n_{2}+\mathbb{R}n_{3}$
with
\begin{equation}
  \begin{split}
 &n_{3}^{2},\: n_{1}n_{3},\, n_{2}n_{3}\in \J^{3}=0,\\
&n_{1}^{2}=\alpha n_{3},\, n_{2}^{2}=\beta n_{3},\, n_{1}n_{2}=\gamma n_{3}.
  \end{split}\label{produto}
\end{equation}
If $(\J,\cdot)$ is an algebra with basis $\{N_1,N_2,N_3\}$ whose multiplication $\cdot$ satisfies \eqref{produto} then for any $\alpha,\beta,\gamma\in\mathbb{R}$, $\J$ is a Jordan algebra, but at least one of them has to be not null. We claim that $\J$ is one of the following algebras:

\begin{longtable}{|>{\centering}m{0.32cm}|>{\centering}m{5.3cm}|>{\centering}m{1.03cm}|>{\centering}m{1.04cm}|>{\centering}m{1.8cm}|}

\hline 
$\J$ & Multiplication Table & $\dim$ $\operatorname*{Aut}(\J)$ & $\dim$ $\operatorname*{Ann}(\J)$ & Observation\tabularnewline
\endfirsthead
$\J$ & Multiplication Table & $\dim$ $\operatorname*{Aut}(\J)$ & $\dim$ $\operatorname*{Ann}(\J)$ & Observation\tabularnewline
\endhead
\hline 
$\T_{24}$ & $n_{1}^{2}=n_{2}^{2}=n_{3}$ & $4$ & $1$ & associative\tabularnewline
\hline 
$\T_{25}$ & $\B_{3}\oplus\mathbb{R}n_{3}$ & $5$ & $2$ & associative \tabularnewline
\hline 
$\T_{26}$ & $n_{1}n_{2}=n_{3}$ & $4$ & $1$ & associative\tabularnewline
\hline

\caption{Three-dimensional nilpotent real Jordan algebras of nilpotency type $(2,1)$.}
\end{longtable}

Indeed, let $\beta=0$ then, if $\gamma\neq0$ we obtain that $\T\simeq\T_{26}$
where the isomorphism is given by $N_{1}\mapsto n_{1}+\frac{\alpha}{2\gamma}n_{2}$, $N_2\mapsto n_2$ and $N_{3}\mapsto\gamma^{-1}n_{3}$. 
If $\gamma=0$, necessarily $\alpha\neq0$ and 
$\T\simeq\T_{25}$ with change of basis: $N_1\mapsto n_1$,
$N_{2}\mapsto n_{3}$ and $N_{3}\mapsto\alpha^{-1}n_{2}$. 

On the other hand, denote $-\alpha\beta+\gamma^{2}$ by $\Delta$ and suppose $\beta\neq0$, then:
\begin{enumerate}[label=\roman*]
\item If $\Delta>0$, again $\T\simeq\T_{26}$ via $N_{1}\mapsto\left(\frac{1}{2}+\frac{\gamma}{2\sqrt{\Delta}}\right)n_{1}+\left(-\frac{1}{2}+\frac{\gamma}{2\sqrt{\Delta}}\right)n_{2}$,
$N_{2}\mapsto\frac{\beta}{2\sqrt{\Delta}}n_{1}+\frac{\beta}{2\sqrt{\Delta}}n_{2}$
and $N_{3}\mapsto\frac{\beta}{2\Delta}n_{3}$.

\item If $\Delta<0$, in this case $\T\simeq\T_{24}$ via 
$N_{1}\mapsto n_{1}+\frac{\gamma}{\sqrt{-\Delta}}n_{2}$, $N_{2}\mapsto\frac{\beta}{\sqrt{-\Delta}}n_{2}$
and $N_{3}\mapsto-\frac{\beta}{\Delta}n_{3}$. 

\item If $\Delta=0$ then $\T\simeq\T_{25}$ in both cases: if $\gamma=0$ with
isomorphism $N_{1}\mapsto n_{3}$, $N_{2}\mapsto n_{1}$
and $N_{3}\mapsto\beta^{-1}n_{2}$, and if $\gamma\neq0$ with isomorphism given by $N_{1}\mapsto n_{1}+n_{3}$,
$N_{2}\mapsto\alpha^{-1}\gamma n_{1}$ and $N_{3}\mapsto\alpha^{-1}n_{2}$.
\end{enumerate}

\subsection{Remarks\label{sec:Remarks}}

 We prove that given $\T$ a three-dimensional Jordan algebra over $\mathbb{R}$ then $\T$ is one of the algebras $\T_1$ to $\T_{26}$. To complete the algebraic classification just remains to prove that all them are pairwise non-isomorphic. Comparing the algebra invariants,
namely $\dim\Rad(\T)$, $\dim\operatorname*{Ann}(\T)$, $\dim\operatorname*{Aut}(\T)$, nilpotency type of $\Rad(\T)$, together with properties whether $\J$ is indecomposable,
associative, non-associative, unitary, one needs only to verify whether
there exist isomorphisms between $\T_3$, $\T_4$ and $\T_5$ and between $\T_{24}$ and $\T_{26}$. First, we observe that the algebras $\T_3,\T_4$ and $\T_5$ are associated to non-degenerate symmetric bilinear forms that are non-isomorphic and therefore they are pairwise non-isomorphic.
Finally, the two-dimensional null algebra $\mathbb{R}n_1\oplus\mathbb{R}n_2$ is a subalgebra of $\T_{26}$ but is not a subalgebra of $\T_{24}$, thus $\T_{24}\not\simeq\T_{26}$.

\section{Geometric Classification\label{Geometric Classification}}
 In this section we will determine the geometric classification of three-dimensional real Jordan algebras. But, due to item \ref{direct_sum} of Proposition \ref{conditions} which gives a sufficient condition of existence of deformation between decomposable algebras, first we will describe the varieties of Jordan algebras of dimension less than three.
 The only rigid one-dimensional real Jordan algebra is the simple one 
$\mathbb{R}e$ and it is clear that  $\mathbb{R}e\to\mathbb{R}n$. Therefore, $\Jor_{1}$ is an irreducible algebraic variety of dimension $1$ with two $\G$-orbits.
In \cite{dim2r} the authors proved that $\Jor_2$ is an algebraic variety with $7$ orbits under the action of $\G$ and $3$ irreducible components given by the Zariski closure of the orbits of the algebras $\mathbb{R}e_1\oplus\mathbb{R}e_2$, $\B_2$ and $\B_4$. The deformations between the algebras in $\Jor_2$ are represented in Figure \ref{fig:orbitasJor2Reais}.

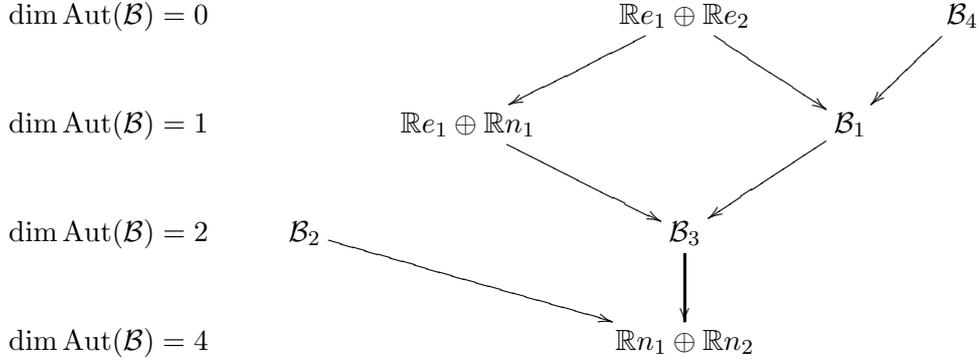
\begin{figure}[H]
\[
\hspace{-1cm}\xymatrix{\dim\Aut(\B)=0 &  &  & \mathbb{R}e_{1}\oplus\mathbb{R}e_{2}\ar[dl]\ar[dr] &  & \B_{4}\ar[ld]\\
\dim\Aut(\B)=1 &  & \mathbb{R}e_{1}\oplus\mathbb{R}n_{1}\ar[dr] &  & \B_{1}\ar[dl]\\
\dim\Aut(\B)=2 & \B_{2}\ar[drr] &  & \B_{3}\ar[d]\\
\dim\Aut(\B)=4 &  &  & \mathbb{R}n_{1}\oplus\mathbb{R}n_{2}
}
\]

\caption{\label{fig:orbitasJor2Reais}Complete description of the $\G$-orbits of  $\Jor_{2}$ }
\end{figure}
Now, we are ready to determine the irreducible components of the variety of three dimensional real  Jordan algebras.

\begin{theorem}

\label{thm:teoprincipalemR}The variety $\Jor_{3}$ of three-dimensional real Jordan algebras is a connected affine variety of dimension $9$ with $26$ orbits under the action of $\operatorname{GL}(\V)$ and $8$ irreducible components given by Zariski closure of the orbits of the following algebras:

\[
\Omega=\left\{ \T_{1},\T_{2},\T_{3},\T_{4},\T_{5},\T_{7},\T_{12},\T_{20}\right\}.
\]
\end{theorem}

\begin{proof}
For any $\T\in\Jor_3$, the change of basis
\[
g(t)=\begin{bmatrix}t & 0 & 0 \\
0 & t & 0 \\
0 & 0 & t 
\end{bmatrix},
\]
gives $\T\to\T_{22}$ then the orbit $\T_{22}^{\G}$ belongs to any irreducible component of $\Jor_3$ and hence 
$\Jor_3$ is a connected affine variety. The fact that it has $26$ $\G$-orbits follows of the algebraic classification in Section \ref{sec:3-Jordan-algebras}. Then, 
$\Jor_{3}=\bigcup_{i=1}^{26} \T_{i}^{\G}$ is a finite union of orbits which are locally closed sets, thus 
\[
\dim\Jor_{3}=\max_{1\leq i\leq26}\left\{ \dim\T_{i}^{\G}\right\} =\dim\T_{1}^{\G}=3^{2}-\dim\Aut(\T_{1})=9.
\]

We will divide the rest of the proof in two parts. Firstly, we will show that all algebras in $\Omega$ are rigid and then, we will prove that there is no other rigid algebra in $\Jor_3$, that is, every structure from $\T_{1}$ to $\T_{26}$ in the algebraic 
classification is dominated by one of the algebras from $\Omega$.

The algebras $\T_{1}$ and $\T_{2}$ have $\dim Aut(\T_{i})=0$ then,
by Proposition \ref{conditions}\ref{Aut}, no other algebra in $\Jor_3$ can be a deformation of them. Therefore $\T_{1}$ and $\T_{2}$ are rigid algebras. By the same argument the only algebras that could be a deformation of $\T_{3}$,
$\T_{4}$ and $\T_{5}$ are $\T_{1}$ and $\T_{2}$ but,
by Proposition \ref{conditions}\ref{identidade}, an associative algebra could not be a deformation of a non-associative one. Thus $\T_{3}$,
$\T_{4}$ and $\T_{5}$ also are rigid algebras.

There is no algebra in $\Jor_{3}$ which dominates $\T_{7}$: by 
Proposition \ref{conditions}\ref{Rad} the only possible candidates to be deformations of $\T_{7}$
are those whose $\dim\Rad(\T_{i})\leq1$, that is, the algebras $\T_{1}$ to $\T_{11}$. By Proposition
\ref{conditions}\ref{identidade}, we may exclude $\T_{1}$,
$\T_{2}$, $\T_{6}$, $\T_{9}$ and $\T_{10}$ from the list. Since $\dim\Aut(\T_{i})=2$,
for $i=7,8,11$ neither $\T_{8}$ or $\T_{11}$ could dominate $\T_{7}$ due to Proposition \ref{conditions}\ref{Aut}. The dimension of the $2$-cocycle groups of the three remainder algebras, $\T_{i}$ with $i=3,4,5$, is $8$ while $\dim Z^{2}(\T_{7},\T_{7})=7$,
thus  by Proposition \ref{conditions}\ref{dimZ2} none of them dominates $\T_{7}$. We conclude that $\T_{7}$ is rigid.

It follows from Example \ref{T12rigida} that the algebra $\T_{12}$ is rigid.

For the proof of the rigidness of $\T_{20}$ we use the same arguments as for $\T_{7}$: Since $\dim\Rad(\T_{20})=2$ no nilpotent algebra, i.e. $\T_{i}$ for  $22\leq i\leq26$, is a deformation of $\T_{20}$. By the argument of dimension of the automorphism group, the algebras $\T_{i}$ for $i=7,8,11,\cdots,21$ do not dominate 
$\T_{20}$. Also we can exclude from the list the associative algebras 
$\T_{1}$, $\T_{2}$, $\T_{6}$, $\T_{9}$ and $\T_{10}$ because 
$\T_{20}$ is non-associative.  Finally, $\dim Z^{2}(\T_{20},\T_{20})=7$
implies that $\T_{i}$ for $i=3,4,5$ are not deformations of $\T_{20}$ and 
therefore the algebra $\T_{20}$ is rigid.

It remains to show that for any algebra $\T_{i}$ for $i=1,\cdots,26$ there is an algebra
$\T\in\Omega$ such that $\T$ is a deformation of $\T_{i}$. In what follows all transformations are given using the basis from Section \ref{sec:3-Jordan-algebras}.

Firstly, the orbits of the algebras $\T_{6}$, $\T_{21}$ and $\T_{24}$,
are contained in $\overline{\T_{1}^{\G}}$: combining the deformation between two-dimensional real Jordan algebras obtained in \cite{dim2r} with
Proposition \ref{conditions}\ref{direct_sum} we have: $\T_{1}\to\T_{6}$ and $\T_{6}\to\T_{21}$. If we consider the family of automorphisms of $\T_{21}$
given by $A_{t}=tn_{1}$, $B_{t}=t^{2}e_{1}-n_{2}$ and $C_{t}=t^{2}n_{2}$, 
making $t$ tends to $0$ we obtain the algebra $\T_{24}$, thus $\T_{21}\rightarrow\T_{24}$.

The orbits of the  $\T_{9}$, $\T_{10}$, $\T_{13}$, $\T_{15}$,
$\T_{18}$, $\T_{23}$, $\T_{25}$, and $\T_{26}$, belong to
$\overline{\T_{2}^{\G}}$: again, combining the results for dimension two in \cite{dim2r} with Proposition \ref{conditions}\ref{direct_sum} we obtain: $\T_{2}\to\T_{9}$, $\T_{2}\to\T_{10}$
and $\T_{10}\to\T_{15}$. From \cite{irynashesta} we get the following deformations over $\mathbb{R}$: $\T_{9}\to\T_{18}$, $\T_{18}\to\T_{13}$, 
$\T_{23}\to\T_{26}$ and $\T_{26}\to\T_{25}$. To see that $\T_{15}\to\T_{23}$ take, for $t\neq0$, the change of basis $A_{t}=t^2 n_{2}$, $B_{t}=t n_{1}-t n_2$ and $C_{t}=te_{1}+n_1+n_2$ of $\T_{15}$. Since $B_t C_t=A_t+t B_t$ and ${C_t}^2=B_t+t C_t$ we get the structure of $\T_{23}$ when $t$ tends to zero.

The rigid algebra $\T_{3}$ dominates $\T_{8}$ and $\T_{14}$: consider,
for $t\neq0$, the family of automorphisms of $\T_{3}$
given by $A_{t}=\frac{1}{2}e_{1}+\frac{1}{2}e_{2}$, $B_{t}=\frac{1}{2}e_{1}-\frac{1}{2}e_{2}$
and $C_{t}=te_{3}$ when we make $t$ tends to $0$ we obtain 
$\T_{8}$, thus $\T_{3}\rightarrow\T_{8}$. From \cite{irynashesta}
we get $\T_{8}\rightarrow\T_{14}$.

The algebra $\T_{11}$ belong to the Zariski closure of the orbit of $\T_{4}$. To show that consider, for $t\neq0$, the change of basis of $\T_{4}$: $A_{t}=e_{1}$, $B_{t}=e_{2}$ and $C_{t}=te_{3}$ when
$t$ tends to $0$ we obtain the algebra $\T_{11}$. From \cite{irynashesta} and \cite{dim2r} it follows that 
 $\T_{20}\to\T_{16}$ and $\T_{7}\to\T_{17}$, respectively. 

Lastly, it remains to show that the rigid algebra $\T_{5}$ dominates $\T_{19}$. For $t\neq0$, it is sufficient to consider the change of basis: $A_{t}=\frac{1}{2}e_{1}-\frac{1}{2}e_{2}$, $B_{t}=t e_{3}$
and $C_{t}=\frac{t^2}{2}e_{1}+\frac{t^2}{2}e_2$ of $\T_5$. Since ${B_t}^2=C_t+t^2A_t$, $B_tC_t=\frac{t^2}{2}B_t$ and ${C_t}^2=t^2C_t$ the structural constants of $\T_5$ tends to those of $\T_{19}$ when $t$ tends to zero. This completes the geometric description of the variety $\Jor_3$. The orbits of $\Jor_{3}$ are represented in  Figure \ref{fig1}.
\end{proof}

\psfrag{dim Aut(T)}{$\dim\Aut(\T)$} \psfrag{T11}{$\T_{11}$}
\psfrag{T12}{$\T_{12}$} \psfrag{T13}{$\T_{13}$} \psfrag{T14}{$\T_{14}$}
\psfrag{T15}{$\T_{15}$} \psfrag{T16}{$\T_{16}$} \psfrag{T17}{$\T_{17}$}
\psfrag{T18}{$\T_{18}$} \psfrag{T19}{$\T_{19}$} \psfrag{T20}{$\T_{20}$}
\psfrag{T1}{$\T_{1}$} \psfrag{T2}{$\T_{2}$} \psfrag{T3}{$\T_{3}$}
\psfrag{T4}{$\T_{4}$} \psfrag{T5}{$\T_{5}$} \psfrag{T6}{$\T_{6}$}
\psfrag{T7}{$\T_{7}$} \psfrag{T8}{$\T_{8}$} \psfrag{T9}{$\T_{9}$}
\psfrag{T10}{$\T_{10}$}\psfrag{T21}{$\T_{21}$}\psfrag{T22}{$\T_{22}$}\psfrag{T23}{$\T_{23}$}\psfrag{T24}{$\T_{24}$}\psfrag{T25}{$\T_{25}$}\psfrag{T26}{$\T_{26}$}

\begin{figure}[H]
\centering{}\includegraphics[width=0.74\textwidth]{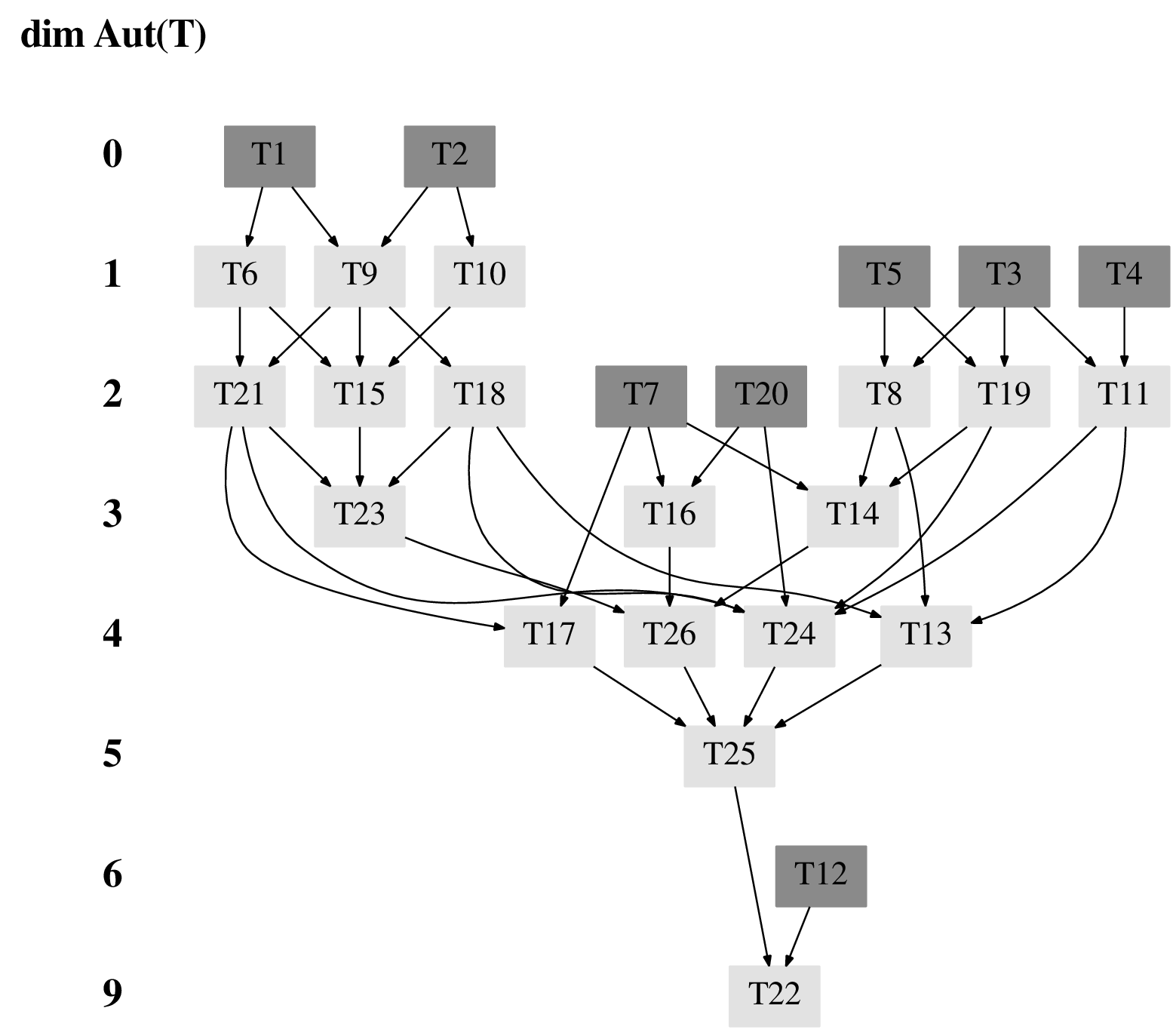}\caption{Description of the $\G$-orbits of the variety $\Jor_{3}$\label{fig1}}
\end{figure}

\subsection{Final Remarks}
 We do not provide the complete list of deformations between algebras in $\Jor_3$. For any non-rigid algebra $\T$ we found at least one rigid 
which dominates $\T$, for further examples of deformation  in $\Jor_3$ see \cite{teseJenny}.

We note that any algebra $\J$ in $\Omega$  satisfies the sufficient condition for rigidness of Proposition \ref{cohomology}, namely $H^2(\J,\J)=0$. 

The informations known until the moment about the number of orbits and irreducible components of the variety $\Jor_n$ when the base field is $\mathbb{R}$ or $\mathbb{C}$ are reunited in Table \ref{comparison}.     

\begin{table}[H]
\begin{centering}
\begin{tabular}{|c|c|c|c|c|}
\hline 
 & \multicolumn{2}{c|}{$\Jor_{n}^{\mathbb{R}}$} & \multicolumn{2}{c|}{$\Jor_{n}^{\mathbb{C}}$}\tabularnewline
\hline 
$n$ & No. orbits & No. components & No. orbits & No. components\tabularnewline
\hline 
$1$ & $2$ & $1$ & $2$ & $1$\tabularnewline
\hline 
$2$ & $7$ & $3$ & $6$ & $2$\tabularnewline
\hline 
$3$ & $26$ & $8$ & $20$ & $5$\tabularnewline
\hline 
$4$ & $>109$ & $\geq18$ & $73$ & $10$\tabularnewline
\hline 
$5$ & - & - & $\gg223$ & $\geq26$\tabularnewline
\hline 
\end{tabular}
\par\end{centering}
\caption{\label{comparison}Comparison of varieties $\Jor_{n}^{\mathbb{C}}$ and $\Jor_{n}^{\mathbb{R}}$}
\end{table}

\newpage

\bibliographystyle{dmm}
\bibliography{library}

\end{document}